\documentclass[11pt]{article}
\usepackage{amsmath,amssymb,amsfonts,amsxtra}
\usepackage{amsthm}
\usepackage{verbatim}
\usepackage{times}
\usepackage{enumerate}
\usepackage{url}

\newtheorem{thm}{Theorem}

\newtheorem{lemma}[thm]{Lemma}

\numberwithin{equation}{section}

\newcommand{\Comment}[1]{}
\newcommand{\rbr}[1]{\left( {#1} \right)}

\newcommand{\cbr}[1]{\left\{ {#1} \right\}}
\newcommand{\abr}[1]{\left\langle {#1} \right\rangle}

\newcommand{\eq}[1]{(\ref{#1})}

\usepackage{environ}
\NewEnviron{as}{\begin{align}\begin{split}\BODY\end{split}\end{align}}
\NewEnviron{as*}{\begin{align*}\begin{split}\BODY\end{split}\end{align*}}

\def\CC{\mathbb{C}}
\def\RR{\mathbb{R}}

\def\ZZ{\mathbb{Z}}
\def\NN{\mathbb{N}}

\def\supp{\text{supp}}

\begin{document}
\title{Explicit Salem Sets in $\mathbb{R}^2$}
\author{Kyle Hambrook}
\date{May 19, 2016}
\maketitle

\begin{abstract}
We construct explicit (i.e., non-random) examples of Salem sets in $\mathbb{R}^2$ of dimension $s$ for every $0 \leq s \leq 2$. In particular, we give the first explicit examples of Salem sets in $\mathbb{R}^2$ of dimension $0 < s < 1$. This extends a theorem of Kaufman.
\end{abstract}

\section{Introduction}\label{intro}

\subsection{Basic Notation} 

For $x \in \mathbb{R}^d$, $|x| = |x|_{\infty} = \max_{1 \leq i \leq d} |x_i|$ and $|x|_2 = (\sum_{i=1}^{d} |x_i|^2)^{1/2}$. 
For $x,y \in \RR^d$, $\abr{x,y} = \sum_{i=1}^{d} x_i y_i$ is the Euclidean inner product.
If $A$ is a finite set, $|A|$ is the cardinality of $A$. 
The expression $a \lesssim b$ stands for ``there is a constant $c > 0$ such that $a \leq cb$.'' 
The expression $a \gtrsim b$ is analogous. 

\subsection{Background}

If $\mu$ is a finite Borel measure on $\mathbb{R}^d$, then the Fourier transform of $\mu$ is defined by
$$
\widehat{\mu}(\xi) = \int_{\mathbb{R}^d} e^{-2 \pi i \abr{\xi,x}} d\mu(x) \quad \forall \xi \in \mathbb{R}^d.
$$

It is a classic result essentially due to Frostman \cite{Frostman} 
that the Hausdorff dimension of any Borel set $A \subseteq \mathbb{R}^d$ can be expressed as 
$$
\dim_H A = \sup\cbr{s \in [0,d] : \int_{\mathbb{R}^d} |\widehat{\mu}(\xi)|^2|\xi|^{s-d} d\xi  < \infty \text{ for some } \mu \in \mathcal{P}(A)}, 
$$
where $\mathcal{P}(A)$ denotes the set of all Borel probability measures with compact support contained in $A$.

The Fourier dimension of a set $A \subseteq \mathbb{R}^d$ is defined to be
$$
\dim_F A = \sup\cbr{s \in [0,d] : \sup_{0 \neq \xi \in \mathbb{R}^d}|\widehat{\mu}(\xi)|^2 |\xi|^{s} < \infty \text{ for some } \mu \in \mathcal{P}(A)}.
$$

As general references for Hausdorff and Fourier dimension, see \cite{Falconer}, \cite{mattila-book}, \cite{mattila-book-2}, \cite{Wolff}. Recent papers by Ekstr\"{o}m, Persson, and Schmeling \cite{ek-1} and Fraser, Orponen, and Sahlsten \cite{FOS} have revealed some interesting subtleties about Fourier dimension. 

Plainly, for every Borel set $A \subseteq \mathbb{R}^d$, 
$$
\dim_F A \leq \dim_H A.
$$

Every $k$-dimensional plane in $\mathbb{R}^d$ with $k < d$ has Fourier dimension $0$ and Hausdorff dimension $k$. 
The middle-thirds Cantor set in $\mathbb{R}$ has Fourier dimension $0$ and Hausdorff dimension $\ln 2 / \ln 3$. 
K{\"o}rner \cite{Korner} has shown that for every $0 \leq s \leq t \leq 1$ there is a compact set $A \subseteq \mathbb{R}$ with Fourier dimension $s$ and Hausdorff dimension $t$.

Sets $A \subseteq \mathbb{R}^d$ with $$\dim_F A = \dim_H A$$ are called Salem sets. 

Every ball in $\mathbb{R}^d$ is a Salem set of dimension $d$. Every countable set in $\mathbb{R}^d$ is a Salem set of dimension zero. Less trivially, every sphere in $\mathbb{R}^d$ is a Salem set of dimension $d-1$. 
%
Salem sets in $\mathbb{R}^d$ of dimension 
$s \neq 0,d-1,d$   
are more exotic.

There are many random constructions of Salem sets. 
Using Cantor sets with randomly chosen contraction ratios, Salem \cite{Salem} was the first to show that for every $s \in (0,1)$ there is a Salem set in $\mathbb{R}$ of dimension $s$. 
Kahane showed that images of compact subsets of $\mathbb{R}^d$ under certain stochastic processes (namely, Brownian motion, fractional Brownian motion, and Gaussian Fourier series) are almost surely Salem sets (see {\cite{Kahane-1966-Brownian}, \cite{Kahane-1966-Fourier}, \cite[Ch.17,18]{Kahane}). 
Through these results, 
Kahane established that for every $s \in (0,d)$ there is a Salem set in $\mathbb{R}^d$ of dimension $s$.  
Ekstr\"{o}m \cite{ek-2} has shown that the image of any Borel set in $\mathbb{R}$ under a random diffeomorphism is almost surely a Salem set. 
Other random constructions of Salem sets have been given by Bluhm \cite{Bluhm-1}, {\L}aba and Pramanik \cite{LP}, Shmerkin and Suomala \cite{shmerkin-suomala}, and Chen and Seeger \cite{chen-seeger2015}.

These random constructions give collections of sets where each individual set is ``almost surely'' or ``with positive probability" a Salem set. But they don't provide any explicit examples of Salem sets. 

Explicit Salem sets are much more rare. Kaufman \cite{Kaufman} gave the first explicit examples of Salem sets in $\mathbb{R}$ of arbitrary dimension $s \in (0,1)$. Kaufman showed that set of $\tau$-well-approximable numbers 
$$
E(\tau) = \cbr{x \in \mathbb{R} : |qx-r| \leq |q|^{-\tau} \text{ for infinitely many } (q,r) \in \ZZ^2}
$$
is a Salem set of dimension $2/(1+\tau)$ when $\tau > 1$. The Hausdorff dimension of $E(\tau)$ was known to be $2/(1+\tau)$ by the classic theorem of Jarn{\'\i}k \cite{Jarnik} and Besicovitch \cite{Bes}. Kaufman showed that the Fourier dimension of $E(\tau)$ is also $2/(1+\tau)$. Note that Dirichlet's approximation theorem easily gives $E(\tau) = \mathbb{R}$ when $\tau \leq 1$.  
K{\"o}rner \cite{Korner} combined Kaufman's construction and a Baire category argument to prove that for every $0 \leq s \leq t \leq 1$ there is a compact set $A \subseteq \mathbb{R}$ with Fourier dimension $s$ and Hausdorff dimension $t$. 
Hambrook \cite{hambrook-explicit} generalized Kaufman's argument to show that many sets in $\mathbb{R}$ closely related to $E(\tau)$ are also Salem sets. 

Bluhm \cite{Bluhm-2} gave a detailed account of what is essentially Kaufman's proof and also pointed out that (as a consequence of a theorem of Gatesoupe \cite{Gatesoupe}) the radial set $\cbr{x \in \mathbb{R}^d : |x|_2 \in E(\tau)}$ is 
a Salem set 
in $\mathbb{R}^d$ of dimension $d-1 + 2/(1+\tau)$ whenever $\tau > 1$. However, explicit Salem sets in $\mathbb{R}^d$ of dimension $0 < s < d-1$ were unknown until now.


From the point of view of Diophantine approximation, the natural mutli-dimensional generalization of $E(\tau)$ is 
$$
E(m,n,\tau) = \cbr{x \in \mathbb{R}^{mn} : |xq - r| \leq |q|^{-\tau} \text{ for infinitely many } (q,r) \in \ZZ^n \times \ZZ^m},
$$
where we identify $\mathbb{R}^{mn}$ with the set of $m \times n$ matrices with real entries. 
By Minkowski's theorem on linear forms, $E(m,n,\tau) = \mathbb{R}^{mn}$ when $\tau \leq n/m$.  
Bovey and Dodson \cite{BD} showed the Hausdorff dimension of $E(m,n,\tau)$ is $m(n-1) + (m+n)/(1+\tau)$ if $\tau > n/m$. The $n=1$ case was done earlier by Jarn{\'\i}k \cite{Jarnik} and Eggleston \cite{Eggleston}. The mass transference principle and slicing technique of Beresnevich and Velani \cite{BV}, \cite{BV-2} may also be used to compute the Hausdorff dimension of $E(m,n,\tau)$. 

Hambrook \cite{hambrook-explicit} proved the Fourier dimension of $E(m,n,\tau)$ is at least $2n/(1+\tau)$ if $\tau > n/m$. However, it is unclear whether $E(m,n,\tau)$ is a Salem set when $\tau > n/m$ and $mn > 1$.

\subsection{Statement of Results}

In the present paper, we extend Kaufman's method \cite{Kaufman} and give explicit examples of Salem sets in $\mathbb{R}^2$ of every dimension $s \in [0,2]$. In particular, we give the first explicit examples of Salem sets in $\mathbb{R}^2$ of dimension $0 < s < 1$.

The key idea is to identify $\RR^2$ with $\CC$. 
Then $\RR^2$ is a field (so we can multiply and divide elements of $\RR^2$), and $\ZZ^2$ is identified with the ring of Gaussian integers $\ZZ+i\ZZ$. 
This allows us to basically follow Kaufman's argument. 

As a reference for the Gaussian integers, see for example \cite{HW}. It will be important that the divisor bound for the Gaussian integers has the same shape as the divisor bound for the integers.

Let $\tau \in \mathbb{R}$. Define 
$$
E_{\ast}(\tau) = \cbr{x \in \mathbb{R}^2 : |qx - r| \leq |q|^{-\tau} \text{ for infinitely many } (q,r) \in \ZZ^2 \times \ZZ^2}.
$$
We identify $\RR^2$ and $\CC$, so $qx$ is a product of complex numbers. 
Note $\abr{\xi, x}$, which appears in the definition of the Fourier transform, is still $\abr{\xi, x} = \xi_1 x_1 + \xi_2 x_2$ for $\xi,x \in \RR^2$.

\begin{thm}\label{main-thm}
For every closed ball $B \subseteq \mathbb{R}^2$, there exists a Borel probability measure with support contained in $E_{\ast}(\tau) \cap B$ such that
$$
|\widehat{\mu}(\xi)| \lesssim |\xi|^{-2/(1+\tau)} \exp(\ln |\xi| / \ln \ln |\xi|) \quad \forall \xi \in \mathbb{R}^2, |\xi| > e.
$$
\end{thm}
The proof of Theorem \ref{main-thm} is an extension of Hambrook's  variation \cite{hambrook-explicit} on Kaufman's argument. 

By a standard covering argument, we have
$
\dim_{H} E_{\ast}(\tau) \leq \min\cbr{4/(1+\tau), 2}.
$
Therefore Theorem \ref{main-thm} implies 
\begin{thm}
$E_{\ast}(\tau)$ is a Salem set with 
$$
\dim_{H} E_{\ast}(\tau)  = \dim_{F} E_{\ast}(\tau) = \min\cbr{ 4/(1+\tau) , 2 }.
$$
\end{thm}

The decay rate in Theorem \ref{main-thm} can actually be improved slightly by adhering more closely to Kaufman's original argument. 
Let $P$ denote the set of Gaussian primes. That is, $P$ is the set of prime elements in the Gaussian integers. Define
$$
E_{\ast}(P,\tau) = \cbr{x \in \mathbb{R}^2 : |qx - r| \leq |q|^{-\tau} \text{ for infinitely many } (q,r) \in P \times \ZZ^2}.
$$
Then $E_{\ast}(P,\tau) \subseteq E_{\ast}(\tau)$ and we have 
\begin{thm}\label{prime-thm}
For every closed ball $B \subseteq \mathbb{R}^d$, there exists a Borel probability measure with support contained in $E_{\ast}(P,\tau) \cap B \subseteq E_{\ast}(\tau) \cap B$ such that
$$
|\widehat{\mu}(\xi)| \lesssim |\xi|^{-2/(1+\tau)} \ln|\xi| \ln \ln |\xi| \quad \forall \xi \in \mathbb{R}^d, |\xi| > e.
$$
\end{thm}

By adapting some ideas of Hambrook \cite{hambrook-explicit} to the present setting, one readily obtains a more general result than Theorem \ref{main-thm}. The statement requires some preparation. 
Let $Q$ be an infinite subset of $\ZZ^2$, let $\Psi: \ZZ^2 \rightarrow [0,\infty)$ be positive on $Q$, and let $\theta \in \mathbb{R}^2$.
Define
$$
E_{\ast}(Q,\Psi,\theta) = \cbr{x \in \mathbb{R}^2 : |qx - r - \theta| \leq \Psi(q) \text{ for infinitely many } (q,r) \in Q \times \ZZ^2}. 
$$
Evidently, $E_{\ast}(\tau) = E_{\ast}(\ZZ^2,q \mapsto |q|^{-\tau},0)$. For $M > 0$, define
$$
Q(M) = \cbr{q \in Q : M/2 < |q| \leq M}, \quad \epsilon(M) = \min_{q \in Q(M)} \Psi(q).
$$
A function $h:(0,\infty) \rightarrow \mathbb{R}$ will be called slowly growing if there is an $M > 0$ such that $h$ is positive and non-decreasing on $[M,\infty)$ and 
${ \displaystyle \lim_{x \rightarrow \infty}} \frac{\ln h(x)}{\ln x} = 0$; the limit is often abbreviated as $h(x) = x^{o(1)}$. There always exists a number $a \geq 0$, a slowly growing function $h:(0,\infty) \rightarrow \mathbb{R}$, and an unbounded set $\mathcal{M} \subseteq (0,\infty)$ such that
$$
|Q(M)|\epsilon(M)^a h(M) \geq M^a \quad \forall M \in \mathcal{M}.
$$
\begin{thm}\label{general-thm}
For every closed ball $B \subseteq \mathbb{R}^d$, there exists a Borel probability measure with support contained in $E_{\ast}(Q,\Psi,\theta) \cap B$ such that 
$$
|\widehat{\mu}(\xi)| \lesssim |\xi|^{-a} \exp( \ln |\xi| / \ln \ln |\xi|) h(4|\xi|) \quad \forall \xi \in \mathbb{R}^2, |\xi| > e.
$$
\end{thm}


The proof of Theorem \ref{main-thm} is divided over sections \ref{part 1}, \ref{part 2}, and \ref{part 3}. 
In section \ref{prime-outline}, we explain how to modify the proof of Theorem \ref{main-thm} to obtain Theorem \ref{prime-thm}. 
We leave the proof of Theorem \ref{general-thm} as an exercise for the reader. It is a simple modification of the proof of Theorem \ref{main-thm} using the ideas of \cite{hambrook-explicit}.


\section{Proof of Theorem \ref{main-thm}: The Function $F_M$}\label{part 1}

For $f:\mathbb{R}^2 \rightarrow \CC$, we will abuse the notation $\widehat{f}$ as follows. 
If $\int_{\mathbb{R}^2} |f(x)| dx < \infty$, then 
$$
\widehat{f}(\xi) = \int_{\mathbb{R}^2} e^{-2 \pi i \abr{\xi, x}} f(x) dx \quad \forall \xi \in \mathbb{R}^2.
$$
If $\int_{[0,1]^2} |f(x)| dx < \infty$ and $f$ is $\ZZ^2$-periodic, then 
$$
\widehat{f}(\xi) = \int_{[0,1]^2} e^{-2 \pi i \abr{\xi, x}} f(x) dx \quad \forall \xi \in \mathbb{R}^d.
$$
There is no ambiguity because if $\int_{\mathbb{R}^2} |f(x)| dx < \infty$ and $f$ is $\ZZ^2$-periodic, then $\widehat{f} = 0$ under either definition. Remember $\abr{\xi, x} = \xi_1 x_1 + \xi_2 x_2$ for $\xi,x \in \RR^2$, 
even though we have identified $\RR^2$ and $\CC$. 

Define $a = 2/(1+\tau)$. Fix a positive integer $K > 2 + a$. Fix an arbitrary non-negative $C^{K}$ function on $\mathbb{R}^2$ with $\int_{\mathbb{R}^2} \phi(x) dx = 1$ and $\text{supp}(\phi) \subseteq [-1,1]^2$. Since $\phi \in C^{K}_{c}(\mathbb{R}^2)$, 
\begin{align}\label{phi decay}
|\widehat{\phi}(\xi)| \lesssim (1+|\xi|)^{-K} \quad \forall \xi \in \mathbb{R}^2.
\end{align}
For $\epsilon > 0$, define
$$
\phi^{\epsilon}(x) = \epsilon^{-2} \phi(\epsilon^{-1} x) \quad \forall x \in \mathbb{R}^2,
$$
$$
\Phi^{\epsilon}(x) = \sum_{r \in \ZZ^2} \phi^{\epsilon}(x-r) \quad \forall x \in \mathbb{R}^2.
$$
Then $\Phi^{\epsilon}$ is $\ZZ^2$-periodic, non-negative, $C^{K}$, and
$$
\widehat{\Phi^{\epsilon}}(k) = \widehat{\phi^{\epsilon}}(k) = \widehat{\phi}(\epsilon k) \quad \forall k \in \ZZ^2.
$$
Therefore 
$$
\Phi^{\epsilon}(x) = \sum_{k \in \ZZ^2} \widehat{\phi}(\epsilon k) e^{2 \pi i \abr{k,x}}
$$
uniformly for all $x \in \mathbb{R}^2$.    
For $q \in \ZZ^2$, define
$$
\Phi^{\epsilon}_{q}(x) = \Phi^{\epsilon}(qx) \quad \forall x \in \mathbb{R}^2.
$$
\begin{lemma}\label{lemma 1}
For all $\ell \in \ZZ^2$, 
$$
\widehat{\Phi^{\epsilon}_{q}}(\ell)
=
\left\{
\begin{array}{cl}
\widehat{\phi}(\epsilon \ell / \overline{q}) & \text{if } q \in D(\overline{\ell}) \\
0 & \text{otherwise }
\end{array} \right.
$$
where
$$
D(\ell) = \cbr{q \in \ZZ^2 : \ell / q \in \ZZ^2}.
$$
\end{lemma}
\begin{proof} We have
\begin{align*}
\widehat{\Phi^{\epsilon}_{q}}(\ell)
=
\int_{[0,1]^2} e^{-2 \pi i \abr{\ell,x}} \sum_{k \in \ZZ^2} \widehat{\phi}(\epsilon k) e^{2 \pi i \abr{k,qx}} dx
=
\sum_{k \in \ZZ^2} \widehat{\phi}(\epsilon k) \int_{[0,1]^2} e^{ 2\pi i (  \abr{k, qx} - \abr{\ell, x}  )   } dx.
\end{align*}
But
\begin{align*}
\abr{k, qx} - \abr{\ell, x} 
&= k_1(qx)_1 + k_2(qx)_2 - \ell_1 x_1 - \ell_2 x_2  \\
&= k_1(q_1x_1 - q_2x_2) + k_2(q_1x_2 + q_2x_1) - \ell_1 x_1 - \ell_2 x_2  \\
&= (k_1q_1 + k_2q_2 - \ell_1)x_1 + (k_2q_1 - k_1q_2 - \ell_2 )x_2 \\
&= ((k\overline{q})_1 - \ell_1) x_1 + ((k\overline{q})_2 - \ell_2)x_2 \\
&= (k\overline{q} - \ell)_1 x_1 + (k\overline{q} - \ell)_2x_2. 
\end{align*}
Therefore
\begin{align*}
\widehat{\Phi^{\epsilon}_{q}}(\ell)
=
\sum_{k \in \ZZ^2} \widehat{\phi}(\epsilon k) 
\int_{[0,1]} \exp( 2\pi i (k\overline{q} - \ell)_1 x_1 ) dx_1 
\int_{[0,1]} \exp( 2\pi i (k\overline{q} - \ell)_2 x_2 ) dx_2. 
\end{align*}
The product of the integrals is $1$ if $\ell / \overline{q} = k$ and is $0$ otherwise.
So $\widehat{\Phi^{\epsilon}_{q}}(\ell) = \widehat{\phi}(\epsilon \ell /\overline{q})$ if $\ell / \overline{q} \in \ZZ^2$ and $\widehat{\Phi^{\epsilon}_{q}}(\ell) = 0$ otherwise. 
Note that $\ell / \overline{q} \in \ZZ^2$ if and only if $\overline{\ell} / q \in \ZZ^2$. 
\end{proof}

For $M > 0$, define 
$$
\ZZ^2(M) = \cbr{q \in \ZZ^2 : M/2 < |q| \leq M},
\quad
\epsilon(M) = \frac{1}{2}M^{-\tau},
$$ 
and
$$
F_M(x) = \frac{1}{|\ZZ^2(M)|} \sum_{q \in \ZZ^2(M)} \Phi^{\epsilon(M)}_{q}(x) \quad \forall x \in \mathbb{R}^2.
$$
Then $F_M$ is $\ZZ^2$-periodic, non-negative, $C^{K}$, and (by Lemma \ref{lemma 1}) 
\begin{align}\label{rewrite}
\widehat{F_M}(\ell) = \frac{1}{|\ZZ^2(M)|} \sum_{q \in \ZZ^2(M) \cap D(\overline{\ell})} \widehat{\phi}(\epsilon(M) \ell / \overline{q}) \quad \forall \ell \in \ZZ^2.
\end{align}
Since $\widehat{\phi}(0) = \int_{\mathbb{R}_d} \phi(x) dx = 1$, we have
\begin{align}\label{F-1}
\widehat{F_{M}}(0) = 1,
\end{align}
and consequently
\begin{align}\label{F-2}
|\widehat{F_{M}}(\ell)| \leq 1 \quad \forall \ell \in \mathbb{Z}^2.
\end{align}
Suppose $\ell \in \ZZ^2$ with $\ell \neq 0$. If $q \in \ZZ^2(M) \cap D(\overline{\ell})$, then $M/2 < |q|_2$ and $|\overline{\ell}/q|_2 \geq 1$, which implies $|\ell|_2 > M/2$. So if $|\ell|_2 \leq M/2$, then the sum in \eq{rewrite} is empty and $\widehat{F_M}(\ell) = 0$. Note $|\ell| \leq M/4$ implies $|\ell|_2 \leq M/2$.
Therefore
\begin{align}\label{F-3}
\widehat{F_M}(\ell) = 0 \quad \forall \ell \in \ZZ^2, 0 < |\ell| \leq M/4.
\end{align}

\begin{lemma}\label{lemma 2}
For every $\zeta > \ln 2$  there exists $L_{\zeta} \in \NN$ such that 
$$
|\widehat{F_M}(\ell)| \lesssim |\ell|^{-a} \exp(\zeta \ln |\ell| / \ln \ln |\ell|) \quad \forall \ell \in \ZZ^2, |\ell| \geq L_{\zeta}.
$$
\end{lemma}
The proof of Lemma \ref{lemma 2} relies on the following divisor bound for the Gaussian integers (see for example \cite{HW}).
\begin{lemma}\label{divisor Gaussian}
For every $\zeta > \ln 2$ there exists $L_{\zeta} \in \NN$ such that 
$$
|D(\ell)| \leq \exp(\zeta \ln |\ell| / \ln \ln |\ell|) \quad \forall \ell \in \ZZ^2, |\ell| \geq L_{\zeta}.
$$
\end{lemma}
\begin{proof}[Proof of Lemma \ref{lemma 2}]
Fix non-zero $\ell \in \ZZ^2$. By \eq{phi decay} and \eq{rewrite}, 
\begin{align*}
|\widehat{F_M}(\ell)| 
&\leq 
\frac{1}{|\ZZ^2(M)|} \sum_{q \in \ZZ^2(M) \cap D(\overline{\ell})} |\widehat{\phi}(\epsilon(M) \ell / \overline{q})| \\
&\lesssim
\frac{1}{|\ZZ^2(M)|} \sum_{q \in \ZZ^2(M) \cap D(\overline{\ell})} (1 + \epsilon(M) |\ell / \overline{q}|)^{-K} \\
&\leq
\frac{|\ZZ^2(M) \cap D(\overline{\ell})|}{|\ZZ^2(M)|}  (1 + (2M)^{-1} \epsilon(M) |\ell|)^{-K}.
\end{align*}
We estimate each factor in the last sum separately. 
Evidently, $|\ZZ^2(M)| \gtrsim M^2$. Since $K \geq a = 2/(1+\tau)$ and $\epsilon(M) = \frac{1}{2}M^{-\tau}$, we have
$$
(1 + (2M)^{-1} \epsilon(M) |\ell|)^{-K} \leq 4^{-a} M^2 |\ell|^{-a}.
$$
Obviously, $|\ZZ^2(M) \cap D(\overline{\ell})| \leq |D(\overline{\ell})| = |D(\ell)|$ . 
So applying Lemma \ref{divisor Gaussian} finishes the proof.
\end{proof}

\begin{lemma}\label{support lemma}
\begin{align}\label{supp 1}
\text{supp}(F_M) \subseteq \cbr{x \in \mathbb{R}^2 : |qx-r| \leq |q|^{-\tau} \text{ for some } (q,r) \in \ZZ^2(M) \times \ZZ^2}.
\end{align}
For any sequence of positive real numbers $(M_k)_{k=1}^{\infty}$ with $M_{k} \leq M_{k+1}/2$ for all $k \in \NN$, we have
\begin{align}\label{supp 2}
\bigcap_{k=1}^{\infty} \text{supp}(F_{M_k}) \subseteq E_{\ast}(\tau).
\end{align}
\end{lemma}
\begin{proof}
Rewrite $F_M$ as 
$$
F_M(x) = \frac{1}{|\ZZ^2(M)|} \sum_{q \in \ZZ^2(M)} \sum_{r \in \ZZ^2} \epsilon(M)^{-2} \phi(\epsilon(M)^{-1} (qx - r)) \quad \forall x \in \mathbb{R}^2.
$$
Suppose $x \in \mathbb{R}^2$ satisfies $F_M(x) > 0$. Since $\phi$ is non-negative and $\supp(\phi) \subseteq [-1,1]$, we must have some $q \in \ZZ^2(M)$ and $r \in \ZZ^2$ such that 
$$
|qx-r| \leq \epsilon(M) = \frac{1}{2}M^{-\tau} \leq \frac{1}{2}|q|^{-\tau}.
$$
More generally, suppose $x \in \text{supp}(F_M)$. Then we can find $x' \in \mathbb{\RR}^2$ such that $F_M(x') > 0$ and $|x-x'|_2 \leq \frac{1}{4}M^{-(1+\tau)}$. Therefore, by the argument above, there is some $q \in \ZZ^2(M)$ and $r \in \ZZ^2$ such that
$$
|qx-r| \leq |qx-qx'| + |qx'-r| \leq \frac{1}{4}M^{-(1+\tau)}|q|_2 + \frac{1}{2}|q|^{-\tau} \leq |q|^{-\tau}.
$$
This proves \eq{supp 1}.

If $x \in \text{supp}(F_{M_k})$ for every $k \in \NN$, we obtain for every $k \in \NN$ a pair $(q^{(k)},r^{(k)}) \in \ZZ^2(M_k) \times \ZZ^2$ with 
$|q^{(k)} x - r^{(k)}| \leq |q^{(k)}|^{-\tau}$. 
The pairs must be distinct because
$$
|q^{(k)}| \leq M_k \leq M_{k+1} / 2 < |q^{(k+1)}| 
\quad \forall k \in \NN.
$$
This proves \eq{supp 2}.
\end{proof}

\section{Proof of Theorem \ref{main-thm}: A Lemma For Recursion}\label{part 2}

\begin{lemma}\label{main-lemma}
For every $\delta > 0$, $M_0 > 0$, and $\chi \in C^{K}_{c}(\mathbb{R}^{2})$, there is an $M_{\ast} = M_{\ast}(\delta,M_0,\chi) \in \NN$ such that $M_{\ast} \geq M_0$ and 
\begin{align*}
|\widehat{\chi F_{M_{\ast}}}(\xi) - \widehat{\chi}(\xi)| \leq \delta g(\xi) \quad \forall \xi \in \mathbb{R}^{2},
\end{align*}
where
$$
g(\xi) = 
\left\{
\begin{array}{cl}
|\xi|^{-a} \exp( \ln |\xi|/\ln \ln |\xi|) & \text{if } \xi \in \mathbb{R}^2, |\xi| > e \\
1 & \text{if } \xi \in \mathbb{R}^2, |\xi| \leq e.
\end{array} \right.
$$
\end{lemma}
The proof will show $M_{\ast}$ can be taken to be any sufficiently large positive number.
\begin{proof}
We begin by recording two auxiliary estimates. Since $\chi \in C^{K}_{c}(\mathbb{R}^{2})$,
\begin{align}\label{108}
|\widehat{\chi}(\xi)| \lesssim (1+|\xi|)^{-K} \quad \forall \xi \in \mathbb{R}^{2}.
\end{align}
For every $p > 2$, we have
\begin{align}\label{108-2}
\sup_{\xi \in \mathbb{R}^{2}} \sum_{\ell \in \ZZ^{2}} (1+|\xi - \ell|)^{-p} < \infty.
\end{align}

Fix $\xi \in \mathbb{R}^2$. We will write $\widehat{\chi F_{M}}(\xi) - \widehat{\chi}(\xi)$ in another form.  Since $F_M$ is $C^{K}$ and $\ZZ^2$-periodic, we have
$$
F_M(x) = \sum_{\ell \in \ZZ^{2}} \widehat{F_M}(\ell) e^{2 \pi i \ell \cdot x} \quad \forall x \in \mathbb{R}^{2}
$$
with uniform convergence. 
Since $\chi \in L^1(\mathbb{R}^{2})$, multiplying by $\chi$ and taking the Fourier transform yields 
\begin{align*}
\widehat{\chi F_M}(\xi) 
= \sum_{\ell \in \ZZ^{2}} \widehat{F_M}(\ell) \int_{\mathbb{R}^{2}} \chi(x) e^{2 \pi i (\ell - \xi) \cdot x} dx
= \sum_{\ell \in \ZZ^{2}} \widehat{F_M}(\ell) \widehat{\chi}(\xi-\ell).
\end{align*} 
Then, by \eq{F-1} and \eq{F-3}, we have 
\begin{align}\label{110-2}
\widehat{\chi F_{M}}(\xi) - \widehat{\chi}(\xi)
=
\sum_{\ell \in \ZZ^{2}} \widehat{\chi}(\xi-\ell) \widehat{F_M}(\ell) - \widehat{\chi}(\xi) 
=
\sum_{|\ell| > M/4} \widehat{\chi}(\xi-\ell) \widehat{F_M}(\ell).
\end{align}

Define $\eta = (K-2-a)/2$, which is positive by our choice of $K$. To estimate $\widehat{\chi F_{M}}(\xi) - \widehat{\chi}(\xi)$, we use \eq{110-2} and consider two cases. 

\textbf{Case 1:} $|\xi| < M/8$. 

If $|\ell| > M/4$, then $|\xi - \ell| > M/8 > |\xi|$. 
Hence by \eq{F-2}, \eq{108}, \eq{108-2}, and \eq{110-2} we have 
\begin{align*}
| \widehat{\chi F_{M}}(\xi) - \widehat{\chi}(\xi) |
&\lesssim
\sum_{|\ell| > M/4} (1+|\xi - \ell|)^{-K}
= 
\sum_{|\ell| > M/4} (1+|\xi - \ell|)^{-a-\eta-(2+\eta)} \\
&\leq (1 + |\xi|)^{-a} (1 + M/8)^{-\eta} \sum_{|\ell| > M/4} (1+|\xi - \ell|)^{-(2+\eta)} 
\leq
\delta g(\xi)
\end{align*}
for all sufficiently large $M$.

\textbf{Case 2:} $|\xi| \geq M/8$. 

Using \eq{110-2}, write 
$$
\widehat{\chi F_{M}}(\xi) - \widehat{\chi}(\xi) 
= \sum_{\substack{|\ell| > M/4 \\ |\ell| \leq |\xi|/2}} \widehat{\chi}(\xi-\ell)  \widehat{F_M}(\ell) 
+ \sum_{\substack{|\ell| > M/4 \\ |\ell| > |\xi|/2}} \widehat{\chi}(\xi-\ell)  \widehat{F_M}(\ell) 
= S_1 + S_2.
$$

If $|\ell| \leq |\xi|/2$, then $|\xi - \ell| \geq |\xi|/2 \geq M/16$. 
Hence by \eq{F-2}, \eq{108}, and \eq{108-2} we have
\begin{align*}
|S_1|
&\lesssim
\sum_{\substack{|\ell| > M/4 \\ |\ell| \leq |\xi|/2}} (1+|\xi - \ell|)^{-K} 
= 
\sum_{\substack{|\ell| > M/4 \\ |\ell| \leq |\xi|/2}} (1+|\xi - \ell|)^{-a-\eta-(2+\eta)} \\
&\leq
(1 + |\xi|/2)^{-a} (1 + M/16)^{-\eta} \sum_{\substack{|\ell| > M/4 \\ |\ell| \leq |\xi|/2}} (1+|\xi - \ell|)^{-(2+\eta)} 
\leq
\frac{1}{2} \delta g(\xi)
\end{align*}
for all sufficiently large $M$. 

Fix $\ln 2 < \zeta < 1$. By Lemma \ref{lemma 2}, \eq{108}, and \eq{108-2} we have
\begin{align*}
|S_2|
&\lesssim
\sum_{\substack{|\ell| > M/4 \\ |\ell| > |\xi|/2}} |\ell|^{-a} \exp\rbr{ {\zeta \ln |\ell|} / {\ln \ln |\ell|} } (1+|\xi - \ell|)^{-K} \\
&\lesssim 
(|\xi|/2)^{-a} \exp\rbr{ {\zeta  \ln (|\xi|/2)} / {\ln \ln (|\xi|/2)} } 
\leq
\frac{1}{2} \delta g(\xi)
\end{align*}
for all sufficiently large $M$.
\end{proof}

\section{Proof of Theorem \ref{main-thm}: The Measure $\mu$}\label{part 3}

Given any closed ball $B \subseteq \mathbb{R}^2$, fix an arbitrary non-negative $C^{K}$ function $\chi_0$ on $\mathbb{R}^2$ with $\text{\supp}(\chi_0) \subseteq B$ and $\int_{\mathbb{R}^2} \chi_0 (x) dx = 1$. Using Lemma \ref{main-lemma}, define
$$
M_1 = M_{\ast}(2^{-2},1,\chi_0), \quad M_{k+1} = M_{\ast}(2^{-k-2},2M_{k},\chi_0 F_{M_1} \cdots F_{M_{k}} ) \quad \forall k \in \NN.
$$
Define measures $\mu_k$ by
$$
d\mu_0 = \chi_0 dx, \quad d\mu_k = \chi_0  F_{M_1} \cdots F_{M_{k}} dx \quad \forall k \in \NN.
$$
By Lemma \ref{main-lemma}, $M_{k} \leq M_{k+1}/2$ for all $k \in \NN$ and 
\begin{align}\label{7}
|\widehat{\mu_k}(\xi) - \widehat{\mu_{k-1}}(\xi)| \leq 2^{-k-1} g(\xi) \quad \forall \xi \in \mathbb{R}^{2}, k \in \NN. 
\end{align}
Since $g$ is bounded, \eq{7} implies $(\widehat{\mu_k})_{k = 0}^{\infty}$ is Cauchy, hence convergent, in the supremum norm. 
Therefore, since each $\widehat{\mu_k}$ is a continuous function, $\displaystyle{\lim_{k \rightarrow \infty}} \widehat{\mu_k}$ is a continuous function.
By \eq{7}, we have
\begin{align}\label{8}
|\lim_{k \rightarrow \infty} \widehat{\mu_k}(\xi) - \widehat{\mu_0}(\xi)| 
\leq 
\sum_{k=1}^{\infty} |\widehat{\mu_k}(\xi) - \widehat{\mu_{k-1}}(\xi)|
\leq \frac{1}{2}g(\xi) \quad \forall \xi \in \mathbb{R}^{2}
\end{align}
Since $\widehat{\mu_0}(0) = \int_{\mathbb{R}^{2}} \chi_0(x)dx = 1$ and $g(0) = 1$, it follows from \eq{8} that 
$$
1/2 \leq |\displaystyle{\lim_{k \rightarrow \infty}} \widehat{\mu_k}(0)| \leq 3/2.
$$ 
Therefore, by L\'{e}vy's continuity theorem, $(\mu_k)_{k=0}^{\infty}$ converges weakly to a non-zero finite Borel measure $\mu$ with $\widehat{\mu} = \displaystyle{\lim_{k \rightarrow \infty}} \widehat{\mu_k}$ and
\begin{align*}
\text{supp}(\mu) 
= \text{supp}(\chi_0) \cap \bigcap_{k=1}^{\infty} \text{supp}(F_{M_k}).
\end{align*}
By Lemma \ref{support lemma} and $\text{supp}(\chi_0) \subseteq B$, we have
$$\text{supp}(\mu) \subseteq B \cap E_{\ast}(\tau).$$
Since $\chi_0 \in C_{c}^{K}(\mathbb{R}^{2})$, we have $\widehat{\mu_0}(\xi) \lesssim (1+|\xi|)^{-a}$ for all $\xi \in \mathbb{R}^{2}$. Combining this with \eq{8} gives 
$$
|\widehat{\mu}(\xi)| \lesssim g(\xi) \quad \forall \xi \in \mathbb{R}^{2}.
$$
By multiplying $\mu$ by a constant, we can make $\mu$ a probability measure. This completes the proof of Theorem \ref{main-thm}.

\section{Outline of Proof of Theorem \ref{prime-thm}}\label{prime-outline}

The proof of Theorem \ref{prime-thm} is obtained by modifying the proof of Theorem \ref{main-thm} in a few places, as we now describe. 

Throughout the proof, we replace $\ZZ^2$ by the set of Gaussian primes $P$, and we replace 
$$\ZZ^2(M) = \cbr{q \in \ZZ^2 : M/2 < |q| \leq M}$$ 
by 
$$P(M) = \cbr{q \in P : M/2 < |q| \leq M}.$$ 

Lemma \ref{lemma 2} is replaced by 
\begin{lemma}\label{prime lemma 2}
$$
|\widehat{F_M}(\ell)| \lesssim |\ell|^{-a} \ln |\ell| \quad \forall \ell \in \ZZ^2, |\ell| \geq 2, M \geq 4.
$$
\end{lemma}
The proof of Lemma \ref{prime lemma 2} is a modification of the proof of Lemma \ref{lemma 2}. Instead of estimating $|\ZZ^2(M)|$ and $|\ZZ^2(M) \cap D(\overline{\ell})|$, we estimate $|P(M)|$ and $|P(M) \cap D(\overline{\ell})|$. By the prime number theorem in the Gaussian integers 
(which is a consequence of Landau's prime ideal theorem), 
we have
$$
|P(M)| \gtrsim \frac{M^2}{\ln M}.
$$
By unique factorization in the Gaussian integers, we have 
$$
|P(M) \cap D(\overline{\ell})| \lesssim \frac{\ln |\ell|}{\ln M}.
$$
We assume $|\ell| \geq 2$ and $M \geq 4$ to avoid technicalities.

Finally, the function $g$ appearing in Lemma \ref{main-lemma} is changed to 
$$
g(\xi) = 
\left\{
\begin{array}{cl}
|\xi|^{-a} \ln|\xi| \ln \ln |\xi| & \text{if } \xi \in \mathbb{R}^2, |\xi| > e \\
1 & \text{if } \xi \in \mathbb{R}^2, |\xi| \leq e.
\end{array} \right.
$$
The estimate for $S_2$ in the proof of Lemma \ref{main-lemma} now goes like this:
By Lemma \ref{prime lemma 2}, \eq{108}, and  \eq{108-2} we have
\begin{align*}
|S_2|
\lesssim
\sum_{\substack{|\ell| > M/4 \\ |\ell| > |\xi|/2}} |\ell|^{-a} \ln|\ell| (1+|\xi - \ell|)^{-K} 
\lesssim 
(|\xi|/2)^{-a} \ln (|\xi|/2) 
\leq
\frac{1}{2} \delta g(\xi)
\end{align*}
for all sufficiently large $M$.


\vspace{1 cm}

\noindent Kyle Hambrook \\
Department of Mathematics, University of Rochester, Rochester, NY, 14627 USA \\
\texttt{khambroo@ur.rochester.edu}

\end{document}